\documentclass[12pt,reqno]{amsart}

\usepackage{amssymb}
\usepackage{dsfont}
\usepackage[mathscr]{euscript}

\usepackage{tikz}
\usetikzlibrary{matrix,decorations.pathreplacing}

\usepackage[ruled,vlined,algosection,norelsize]{algorithm2e}

\usepackage[colorlinks=true]{hyperref}

\usepackage{color}

\DeclareMathOperator{\depth}{depth}
\DeclareMathOperator{\sdep}{sdepth}
\DeclareMathOperator{\hdep}{hdepth}

\DeclareMathOperator{\supp}{supp}
\DeclareMathOperator{\lin}{span}

\DeclareMathOperator{\Ann}{Ann}

\DeclareMathOperator{\Syz}{Syz}

\let\Dirsum=\bigoplus

\let\iso=\cong

\newcommand{\cC}{\mathcal{C}}
\newcommand{\mm}{\mathfrak{m}}

\newcommand{\ZZ}{\mathds{Z}}

\newcommand{\NN}{\mathds{N}}

\newcommand{\KK}{\mathds{K}}

\newcommand{\ab}{\mathbf{a}}
\newcommand{\bb}{\mathbf{b}}
\newcommand{\cb}{\mathbf{c}}
\newcommand{\db}{\mathbf{d}}
\newcommand{\Sb}{\mathbf{s}} 
\newcommand{\gb}{\mathbf{g}}
\newcommand{\eb}{\mathbf{e}}
\newcommand{\nb}{\mathbf{0}}
\newcommand{\Xb}{\mathbf{X}} 

\newcommand{\set}[1]{\{#1\}}
\newcommand{\with}{\,:\,}

\newcommand{\qq}[1]{``#1''}

\newtheorem{coro}[algocf]{Corollary}
\newtheorem{theo}[algocf]{Theorem}
\newtheorem{propo}[algocf]{Proposition}
\newtheorem{conjecture}[algocf]{Conjecture}

\theoremstyle{definition}
\newtheorem{defi}[algocf]{Definition}
\newtheorem{rema}[algocf]{Remark}
\newtheorem{ex}[algocf]{Example}
\newtheorem{question}[algocf]{Question}
\newtheorem{construction}[algocf]{Construction}

\numberwithin{algocf}{section}

\begin{document}

\title{How to compute the Stanley depth of a module}

\author{Bogdan Ichim}

\address{Simion Stoilow Institute of Mathematics of the Romanian Academy, Research Unit 5, C.P. 1-764,
014700 Bucharest, Romania} \email{bogdan.ichim@imar.ro}

\author{Lukas Katth\"an}

\address{Universit\"at Osnabr\"uck, FB Mathematik/Informatik, 49069
Osnabr\"uck, Germany}\email{lukas.katthaen@uos.de}

\author{Julio Jos\'e Moyano-Fern\'andez}

\address{Universitat Jaume I, Campus de Riu Sec, Departamento de Matem\'aticas \& Institut Universitari de Matem\`atiques i Aplicacions de Castell\'o, 12071
Caste\-ll\'on de la Plana, Spain} \email{moyano@uji.es}

\thanks{The first author was partially supported  by the project  PN-II-RU-TE-2012-3-0161, granted by the Romanian National Authority for Scientific Research,
CNCS -- UEFISCDI. The second author was partially supported by
the German Research Council DFG-GRK~1916. The third author was partially supported by the Spanish Government---Ministerio de Econom\'ia y Competitividad (MINECO), grant MTM2012-36917-C03-03.}

\begin{abstract}
In this paper we introduce an algorithm for computing the Stanley depth of a finitely generated multigraded module $M$ over the polynomial ring $\KK[X_1, \ldots, X_n]$.
As an application, we give an example of a module whose Stanley depth is strictly greater than the depth of its syzygy module.
In particular, we obtain complete answers for two open questions raised by Herzog in \cite{H}. Moreover, we show that the question whether $M$ has Stanley depth at least $r$ can be reduced to the question whether a certain combinatorially defined polytope $\mathscr{P}$ contains a $\ZZ^n$-lattice point.
\end{abstract}

\maketitle

\section{Introduction}
Let $\KK$ be a field. Let $R=\KK[X_1, \ldots , X_n]$ be the standard $\ZZ^n$-graded polynomial ring, and let $M=\oplus M_{\ab}$ be a finitely generated $\ZZ^n$-graded $R$-module (also called multigraded in the sequel). The \emph{Stanley depth} of $M$, denoted $\sdep M$, is a combinatorial invariant of $M$ related
to a conjecture of Stanley from 1982 \cite[Conjecture 5.1]{St}, which states that, in the case when $\KK$ is infinite, the inequality $\depth M \leq \sdep M$ holds; this is nowadays called the \emph{Stanley conjecture}. We refer the reader to \cite{intro} for a short introduction to the subject and to \cite{H} for a comprehensive survey.

After the initial submission of this paper, a counterexample to Stanley's conjecture was given by Duval, Goeckner, Klivans, and Martin in \cite{counterexample}.
We would like to mention that the counterexample may be checked directly by computational methods.

The Stanley depth is an interesting invariant which naturally arises in various combinatorial and computational contexts, which remains rather elusive so far, cf.~\cite{Sturmfels1991,Murdock2002,N1,N2}.
Our goal is to answer the following natural question, which was raised by Herzog:
\begin{question}\label{Q:Herzog65}\cite[Question 1.65]{H}
	Does there exist an algorithm to compute the Stanley depth of finitely generated multigraded $R$-modules?
\end{question}

In the particular cases of $R$-modules which are either monomial ideals $I\subset R $, or quotients thereof, this question has been answered by Herzog, Vladoiu, and Zheng \cite{HVZ}. The majority of the published articles concerning Stanley depth are related to this result. A key remark is that, in the cases studied by Herzog, Vladoiu and Zheng, the Hilbert series already determines the module structure. So, the Stanley depth may be computed directly from the Hilbert series of $M$.
This leads to another interesting combinatorial invariant, called the \emph{Hilbert depth} of $M$, which was introduced by Bruns, Krattenthaler, and Uliczka in \cite{BKU}. In fact, the method of \cite{HVZ} extends directly to an algorithm for computing the Hilbert depth of finitely generated multigraded $R$-modules, introduced by the first and third author \cite{IJ} and the first author together with Zarojanu \cite{IZ}.
However, until now, little is known about the computation of the Stanley depth in general.

For computing either the Stanley or the Hilbert depth one has to consider certain combinatorial decompositions. They are called \emph{Stanley decompositions} in the first case, respectively \emph{Hilbert decompositions} in the second case.

An interesting fact is that---beside the interest raised among algebraists and combinatorialists by the conjecture of Stanley---Stanley decompositions have a separate life in applied mathematics.
This goes back to Sturmfels and White \cite{Sturmfels1991}, where it is shown how Stanley decompositions can be used to describe finitely generated graded algebras, e.g. rings of invariants under some group action.
More recently, this found applications in the normal form theory for systems of differential equations with nilpotent linear part (see Murdock \cite{Murdock2002}, Murdock and Sanders \cite{Murdock2007}, Sanders \cite{Sanders2007}).

It is also worth mentioning that, in the particular case of a normal affine monoid, suited Stanley (or Hilbert) decompositions have already been used with success in order to design arguable the fastest available algorithms for computing Hilbert series (see \cite{N1} and \cite{N2}). Further, these algorithms have been used for computing the Hilbert series (and subsequently the associated probability generating functions) corresponding to three well studied (but difficult to compute) voting situations with four candidates arising from the field of social choice: the Condorcet paradox, the Condorcet efficiency of plurality voting and  Plurality versus Plurality Runoff (see \cite{Sch} and \cite{N2} for details).

We remark that every Stanley decomposition is inducing a Hilbert decomposition, but the converse is not true. In fact, in many particular cases studied until now the converse also holds (for example in \cite{HVZ} and the related results). More generally, it makes sense to ask:

\begin{question}\label{Q:2}
	Which Hilbert decompositions are induced by Stanley decompositions?
\end{question}

A precise answer to Question \ref{Q:2} implies an answer to Question~\ref{Q:Herzog65}.
This is the main contribution of the present article:
In Theorem \ref{thm:main} we give an effective criterion to decide whether a given Hilbert decomposition is induced by a Stanley decomposition.
This leads directly to an algorithm for the computation of the Stanley depth of a finitely generated multigraded $R$-module, which we present in Section \ref{Algo}.

Further, as an application of our main result, we are able to construct a counterexample which gives a negative answer to the following open question, also raised by Herzog (see Subsection \ref{subsection:counterex}):
\begin{question}\label{Q:Herzog63}\cite[Question 1.63]{H}\label{q:syz}
	Let $M$ be a finitely generated multigraded $R$-module with syzygy module $Z_k$ for $k = 1,2,\ldots$.
	Is it true that $\sdep Z_{k+1}\ge \sdep Z_k$?
\end{question}

Moreover, we define and study the structure of the set of all \emph{$\gb$-determined Stanley decompositions} of a finitely generated multigraded $R$-module $M$.
In Section \ref{ssec:rado} we show that this set naturally corresponds to the set of solutions of a certain system of linear Diophantine inequalities.
In other words, the question whether $M$ has Stanley depth at least $r$ can be reduced to the question whether a certain combinatorially defined polytope $\mathscr{P}$ contains a $\ZZ^n$-lattice point.
This polytope $\mathscr{P}$ turns out to be an intersection of a certain affine subspace with the positive orthant and \emph{finitely many polymatroids}.

Finally, we would like to point out that we have not been able to either prove or disprove \cite[Conjecture~2]{A1} and \cite[Conjecture 1.64]{H}, despite several computational and theoretical attempts. Further research on these open conjectures is certainly desirable in view of the recent important advance made by Duval, Goeckner, Klivans and Martin in \cite{counterexample}.
\medskip

The article is organized as follows. In Section \ref{Sect:pre}, we fix the notation, recall the definitions and the necessary previous results. In  Section \ref{AlgStanley}, we formulate and prove Theorem \ref{thm:main}, which is the answer to Question \ref{Q:2}. In Section \ref{Algo}, we deduce an algorithm for the computation of the Stanley depth. This fully responds to Herzog's Question \ref{Q:Herzog65}.
In Section \ref{ApplicationsExamples} we answer Question \ref{Q:Herzog63} and we present several interesting applications of the main result.


\section{Prerequisites} \label{Sect:pre}

In this section we recall the basics about both Stanley and Hilbert decompositions.
We refer the reader to \cite{H} for a more comprehensive treatment.
\medskip

Let $\KK$ be a field, $R=\KK[X_1, \ldots ,X_n]$ be the polynomial ring with the fine $\ZZ^n$-grading, and let $M$ be a finitely generated $\ZZ^n$-graded $R$-module.
Throughout the paper, we denote the cardinality of a set $S$ by $|S|$ and we set $[n] := \set{1,\dotsc, n}$.
Moreover, $n$-tuples in $\ZZ^n$ will be denoted by boldface letters as $\ab, \bb, \ldots$, while $a_i$ will denote the $i$-th component of $\ab \in \ZZ^n$. Further, for $\ab\in\NN^n$, we set $\supp(\ab) := \set{i \in [n] \with a_i \neq 0}$ and $\Xb^\ab := X_1^{a_1}\dotsm X_n^{a_n}$.

\begin{defi}
\begin{enumerate}
	\item A \emph{Stanley decomposition} of $M$ is a finite family $(R_i,m_i)_{i\in I}$, in which all $m_i \in M$ are homogeneous and $R_i$ are subalgebras of $R$ generated by a subset of the indeterminates of $R$,
	such that $R_i\cap \Ann m_i=0$ for each $i\in I$, and
	\begin{equation}\label{eq:stadec}
		M=\Dirsum_{i\in I} m_iR_i
	\end{equation}
	as a multigraded $\KK$-vector space.
	\item A \emph{Hilbert decomposition} of $M$ is a finite family $(R_i,\Sb_i)_{i\in I}$, where $\Sb_i\in \ZZ^n$ and the $R_i$ are again subalgebras of $R$ generated by a subset of the indeterminates of $R$ for each $i\in I$, such that
	\begin{equation}\label{eq:hildec}
		M\iso\Dirsum_{i\in I} R_i(-\Sb_i)
	\end{equation}
	as a multigraded $\KK$-vector space.
\end{enumerate}
\end{defi}

Note that every Stanley decomposition $(R_i, m_i)_{i\in I}$ of $M$ gives rise to the Hilbert decomposition $(R_i, \deg m_i)_{i \in I}$.
In the sequel, we will say that a Hilbert decomposition \emph{is induced by a Stanley decomposition} if it arises in this way.
Moreover, observe that, in general, the $R$-module structure of a Stanley decomposition is different from that of $M$, and that Hilbert decompositions depend only on the Hilbert series of $M$, i.e. they do not take the $R$-module structure of $M$ into account.

\begin{defi}
	The \emph{depth} of a Hilbert (resp.~Stanley) decomposition is the minimal dimension of the subalgebras $R_i$ in the decomposition.
	Equivalently, it is the depth of the right-hand side of (\ref{eq:hildec}) (resp.~(\ref{eq:stadec})), considered as $R$-module.
	The \emph{multigraded Hilbert depth} (resp.~the \emph{Stanley depth}) of $M$ is then the maximal depth of a Hilbert (resp.~Stanley) decomposition of $M$.
	We write $\hdep M$ and $\sdep M$ for the multigraded Hilbert resp. Stanley depth.
\end{defi}

We denote by $\preceq$ the componentwise order on $\ZZ^n$ and we set
\[[\ab,\bb] := \set{\cb \in \ZZ^n \with \ab \preceq \cb \preceq \bb}\]
with $\ab,\bb \in \ZZ^n$.

For the computation of the Hilbert resp.~Stanley depth, one may restrict the attention to a certain \emph{finite} class of decompositions.
Let us briefly recall the details.
The module $M$ is said to be {\it positively $\gb$-determined} for $\gb \in \NN^n$ if $M_\ab=0$ for $\ab \notin \NN^n$ and the multiplication
map $\cdot X_k : M_{\ab} \longrightarrow M_{\ab+\eb_k}$ is an isomorphism whenever $a_k \ge g_k$, see Miller \cite{M}.
A characterization of positively $\gb$-determined modules is given by the following:
\begin{propo}\label{prop:ezra}\cite[Proposition 2.5]{M}
	The module $M$ is positively $\gb$-determined if and only if the multigraded Betti numbers of $M$ satisfy $\beta_{0,\ab}^R(M)=\beta_{1,\ab}^R(M)=0$ unless $\ab \in [\nb, \gb]$.
\end{propo}
In particular, if $M$ has no components with negative degrees, then it is always $\gb$-determined for a sufficiently large $\gb\in \NN^n$.
For our purpose, the importance of $M$ being positively $\gb$-determined is that it allows us to restrict the search space for possible Hilbert or Stanley decompositions, as we explain in the following.

For a given Hilbert decomposition $(R_i, \Sb_i)_{i \in I}$ of $M$ and a multidegree $\ab \in \NN^n$, let
\[ \cC(\ab) := \set{i \in I \with (R_i(-\Sb_i))_{\ab} \neq 0 } \]
be the set of indices of those Hilbert spaces that contribute to degree $\ab$.
\begin{propo}\label{lem:gstandart}
	Let $M$ be a positively $\gb$-determined module.
	The following statements are equivalent for a Hilbert decomposition $(\KK[Z_i], \Sb_i)_{i \in I}$ of $M$:
	\begin{enumerate}
		\item $\Sb_i \preceq \gb$ for all $i \in I$.
		\item $\set{j \with (\Sb_i)_j = g_j} \subseteq Z_i$ for all $i \in I$.
		\item $\cC(\ab) = \cC(\ab \wedge \gb)$ for all $\ab \in \NN^n$. (Here, $\ab \wedge \gb$ denotes the componentwise minimum.)
		\item $\bigoplus_{i} \KK[Z_i](-\Sb_i)$ is $\gb$-determined as $R$-module.
	\end{enumerate}
\end{propo}
\begin{proof}
{\bf (1)$\Longrightarrow$(2)}
		Let $i \in I$ and $j \in [n]$ such that $g_j = (\Sb_i)_j$.
		By assumption (1), it holds that $(\Sb_{i'})_{j} \leq g_j = (\Sb_i)_j$ for every $i' \in I$.
		Hence, $i' \in \cC(\Sb_i + \eb_j)$ implies that $X_j \in Z_{i'}$ and $\KK[Z_{i'}](-\Sb_{i'})_{\Sb_i} \neq 0$.
		In particular, $\cC(\Sb_i + \eb_j) \subseteq \cC(\Sb_i)$.
		But $M$ is $\gb$-determined, so
		\[|\cC(\Sb_i + \eb_j)| = \dim_\KK M_{\Sb_i + \eb_j} = \dim_\KK M_{\Sb_i} = |\cC(\Sb_i)| \]
		Thus $i \in \cC(\Sb_i) = \cC(\Sb_i + \eb_j)$ and the claim follows.
		
{\bf (2) $\Longrightarrow$ (3)}
		We first show that $\cC(\ab \wedge \gb) \subseteq \cC(\ab)$ for $\ab \in \NN^n$.
		Let $i \in \cC(\ab \wedge \gb)$.
		It suffices to prove that for each $j \in [n]$ with $(\Sb_i)_j < a_j$, it holds that $j \in Z_i$.
		Note that $(\Sb_i)_j \leq (\ab \wedge \gb)_j$ for every $j$.
		Moreover, if $(\Sb_i)_j < (\ab \wedge \gb)_j$ then $i \in \cC(\ab \wedge \gb)$ implies that $j \in Z_i$.
		On the other hand, if $(\Sb_i)_j = (\ab \wedge \gb)_j$ and $(\Sb_i)_j < a_j$, then $(\Sb_i)_j = g_j$ and thus $j \in Z_i$ by assumption.
		It follows that $\cC(\ab \wedge \gb) \subseteq \cC(\ab)$.
		
		Further, $M$ being $\gb$-determined implies as above that $|\cC(\ab)| = |\cC(\ab \wedge \gb)|$ and thus $\cC(\ab) = \cC(\ab \wedge \gb)$.
		
{\bf (3) $\Longrightarrow$ (1)}
		For each $i \in I$, it holds that $i \in \cC(\Sb_i) = \cC(\Sb_i \wedge \gb)$. Therefore $\Sb_i \preceq \Sb_i \wedge \gb \preceq \gb$.
	
{\bf (1) and (2) $\Longleftrightarrow$ (4)}
		This follows easily by considering the Betti numbers of
$$\bigoplus_{i} \KK[Z_i](-\Sb_i).$$
\end{proof}

Note that the conditions are not equivalent if $M$ is not $\gb$-determined.
Moreover, the existence of a Hilbert decomposition satisfying these conditions for some $\gb \in \NN^n$ does not imply that $M$ is $\gb$-determined.
Motivated by the preceding proposition we introduce the following:
\begin{defi}
\begin{enumerate}
	\item A Hilbert decomposition $\mathfrak{D}$ of $M$ is called \emph{$\gb$-determined} if $M$ is positively $\gb$-determined and $\mathfrak{D}$ satisfies the equivalent conditions of Proposition \ref{lem:gstandart}.
	\item A Stanley decomposition $(R_i, m_i)_{i \in I}$ of $M$ is called \emph{$\gb$-determined} if the underlying Hilbert decomposition $(R_i, \deg m_i)_{i \in I}$ is $\gb$-determined.
\end{enumerate}
\end{defi}

Every Hilbert decomposition of $M$ is $\gb$-determined for a sufficiently large $\gb \in \NN^n$.
On the other hand, for a fixed $\gb \in \NN^n$ there are only \emph{finitely} many $\gb$-determined Hilbert decompositions.
By the following result, it is essentially sufficient to consider $\gb$-determined Hilbert (Stanley) decompositions if $M$ is $\gb$-determined:
\begin{propo}\label{prop:gdet}
	Let $M$ be positively $\gb$-determined.
	\begin{enumerate}
		\item There exists a $\gb$-determined Hilbert decomposition of $M$ whose depth equals the Hilbert depth of $M$.
		\item Similarly, there exists a $\gb$-determined Stanley decomposition of $M$ whose depth equals the Stanley depth of $M$.
	\end{enumerate}
\end{propo}
\begin{proof}
	This is immediate from Corollary 3.4 and Corollary 4.7 of \cite{IJ}, since the decompositions used there are $\gb$-determined.
\end{proof}


\section{Which Hilbert decompositions are induced by Stanley decompositions?}\label{AlgStanley}
In this section we characterize those Hilbert decompositions which are induced by Stanley decompositions.
Throughout the section, we fix a finitely generated $\ZZ^n$-graded $R$-module $M$
and a Hilbert decomposition $\mathfrak{D} = (R_i, \Sb_i)_{i \in I}$ of $M$.
Without loss of generality, we shall assume that both $M$ and $\mathfrak{D}$ are (positively) $\gb$-determined for some $\gb \in \NN^n$.
As above, we set
\[ \cC(\ab) := \set{i \in I \with (R_i(-\Sb_i))_{\ab} \neq 0 } \]
for each multidegree $\ab \in \NN^n$.
Then, Proposition 4.4 of \cite{IJ} may be reformulated as follows:
\begin{propo}\cite[Proposition 4.4]{IJ}\label{hilbert:stanley}
	The given Hilbert decomposition of $M$ is induced by a Stanley decomposition if and only if
	there exist homogeneous elements $(m_i)_{i \in I} \subset M$ with $\deg m_i = \Sb_i$ such that the following holds:
	
	For all $i\in I$ we have that $R_i\cap \Ann m_i=0$, and for all $\ab \in \NN^n, \ab \preceq \gb$, the set
\begin{equation}\label{eq:testset}
\set{ \Xb^{\ab-\Sb_i} m_i \with i \in \cC(\ab)}\subset M_{\ab}
\end{equation}
	is $\KK$-linearly independent.
\end{propo}

\newcommand{\mt}{\tilde{m}}%
The difficulty for applying this result is that one has to choose the right elements $m_i\in M_{\Sb_i}$ in order to determine whether a given Hilbert decomposition is induced by a Stanley decomposition.
In the sequel we present a method for circumventing this problem.
The idea is to consider (for all $i\in I$) \qq{generic} elements $\mt_i \in M_{\Sb_i}$  and to test (for all $\ab \in [\nb,\gb]$) the linear independence of the sets \eqref{eq:testset}  via computations of determinants.
We make this precise in the following manner.

\begin{construction}\label{const}
For the given Hilbert decomposition $(R_i, \Sb_i)_{i \in I}$ of $M$, we construct a collection of matrices $(A_{\ab})_{\ab \in [\nb,\gb]}$ as follows.
First, for each $\ab \in [\nb,\gb]$, we choose a basis $\set{b_{\ab,1}, \dotsc, b_{\ab,l_\ab}}$ for the $\KK$-vector space $M_\ab$. Then,
for each $i \in I$, we set $\mt_i := \sum_{j} Y_{i,j} b_{\Sb_i,j}$ with indeterminate coefficients $Y_{i,1}, \dotsc, Y_{i,l_{\Sb_i}}$.

The matrix $A_{\ab}$ has one row for each of the basis vectors of $M_{\ab}$ and one column for each $i \in \cC(\ab)$.
For every such $i$, expand $\Xb^{\ab - \Sb_i} \mt_i$ in the chosen basis of $M_{\ab}$ and write the coefficients into $A_{\ab}$. More explicitly, if
\[
\Xb^{\ab - \Sb_i}b_{\Sb_i,j}=\sum_{k}c_{j,k}b_{\ab,k}
\]
with $c_{j,k}\in \KK$, then
\[
\Xb^{\ab - \Sb_i} \mt_i = \sum_{k}\big(\sum_{j} c_{j,k} Y_{i,j} \big) b_{\ab,k}.
\]
We set $A_{\ab}=(\sum_{j} c_{j,k} Y_{i,j})_{i,k}$.
For the ease of reference, we also set
\[\tilde{I} := \set{(i,j) \with i \in I, 1\leq j \leq l_{\Sb_i}},\]
so that the entries of $A_{\ab}$ live in the polynomial ring $\KK[Y_{i,j} \with (i,j) \in \tilde{I}]$.
\end{construction}

Note that the entries of $A_{\ab}$ are linear polynomials in the $Y_{i,j}$.
Moreover, the matrices $A_{\ab}$ are square matrices, because the number of rows equals $\dim M_{\ab}$, while the number of columns equals the cardinality of $\cC(\ab)$. But this is also $\dim M_{\ab}$, as we started with a Hilbert decomposition.

\begin{ex}
	We give a simple example to illustrate the construction.
	Let $R = \KK[X_1, X_2]$ and $M = (X_1, X_2) \oplus (X_1X_2) \subset R^2$.
	The module $M$ is positively $\gb$-determined for $g = (1,1)$.
	Let $e_1,e_2$ be the generators of $R^2$.
	We choose as vector space bases $X_1 e_1, X_2 e_1, X_1 X_2 e_1$ and $X_1 X_2 e_2$ for the corresponding components of $M$.
	Consider the Hilbert decomposition
	\[ M \cong R(-1,0) \oplus R(0,-1). \]
	We have $\mt_1 = Y_{1,1} X_1 e_1$ and $\mt_2 = Y_{2,1} X_2 e_1$.
	The matrices $A_{\ab}$ constructed above are in this case
	\begin{equation*}
		A_{(1,0)} = \begin{pmatrix} Y_{1,1} \end{pmatrix}
		\qquad
		A_{(0,1)} = \begin{pmatrix} Y_{2,1} \end{pmatrix}
		\qquad
		A_{(1,1)} = \begin{pmatrix} Y_{1,1} & Y_{2,1} \\ 0 & 0 \end{pmatrix}.
	\end{equation*}
\end{ex}
%
Next theorem is the main result of this paper.
\begin{theo}\label{thm:main}
	With the notation introduced in Construction \ref{const}, the following holds:
	\begin{itemize}
	\item[(a)] Assume $|\KK| = \infty$. Then the given Hilbert decomposition of $M$ is induced by a Stanley decomposition if and only if the determinant of $A_{\ab}$ is not the zero polynomial for all $\ab \in [\nb,\gb]$.
	\item[(b)] Assume $|\KK| = q < \infty$. Let $P := \prod_{\ab \in [\nb,\gb]} \det A_{\ab}$.
		Let further $\tilde{P}$ be the polynomial obtained from $P$ as follows: From every exponent of every monomial in $P$, subtract $q-1$ until the remainder is less than $q$. Then the given Hilbert decomposition of $M$ is induced by a Stanley decomposition if and only if $\tilde{P} \neq 0$.
	\end{itemize}
\end{theo}

\begin{proof}
We use the characterization of Proposition~\ref{hilbert:stanley}.
First, note that the assumption $R_i\cap \Ann m_i=0$ in Proposition \ref{hilbert:stanley} is not really needed:
If the sets
$$
\set{ \Xb^{\ab-\Sb_i} m_i \with i \in \cC(\ab)}
$$
are $\KK$-linearly independent for all $\ab \in [\nb,\gb]$, then the fact that $R_i\cap \Ann m_i=0$ for all $i$ follows automatically.
To see this, assume for the contrary that $R_j\cap \Ann m_j\neq0$ for some $j$. Then there exists a multidegree $\db \in \NN^n$ such that $\Xb^\db m_j = 0$ and $\Xb^\db \in R_j$. But as $M$ is $\gb$-determined, this implies that there exists  $\db'\preceq \db \in \NN^n$ such that $\Xb^{\db'} m_j = 0$ and $\db' + \Sb_j \preceq \gb$ (remember that the multiplication
map $\cdot X_k : M_{\db+\Sb_j-\eb_k} \longrightarrow M_{\db+\Sb_j}$ is an isomorphism if $(\db+\Sb_j)_k > g_k$).
Then the set $\set{ \Xb^{\db'} m_i \with i \in \cC(\db' + \Sb_j)}$ contains the zero vector and therefore cannot be linearly independent.

Next, consider a choice of elements $m_i = \sum_j y_{i,j} b_{\Sb_i,j}$ with $(y_{i,j})_{(i,j) \in \tilde{I}} \subset \KK$.
We now observe that for a fixed $\ab \in \NN^n$, the set $\set{ \Xb^{\ab-\Sb_i} m_i \with i \in \cC(\ab)}$ is $\KK$-linearly independent
if and only if $\det A_{\ab} ((y_{i,j})_{(i,j) \in \tilde{I}}) \neq 0$.
Hence the elements $m_i$ build a Stanley decomposition if and only if $\prod_{\ab \in[\nb,\gb]}\det A_{\ab} ((y_{i,j})_{(i,j) \in \tilde{I}}) \neq 0$.

If the field is infinite, then it is possible to choose such $y_{i,j}$ if and only if $P := \prod_{\ab \in [\nb,\gb]} \det A_{\ab}$ is not the zero polynomial.
This is clearly equivalent to each of the factors $\det A_{\ab}$ being nonzero.

If $\KK$ is finite,
then $P$ has a non-zero value over $\KK^{|\tilde{I}|}$ if and only if it is not contained
in the ideal $(Y_{i,j}^{q} - Y_{i,j} \with (i,j) \in \tilde{I})$.
This set of generators is already a (universal) Gr\"obner basis, hence $P$ is contained in the ideal if and only if its remainder modulo this Gr\"obner basis is zero, see Cox, Little, O'Shea \cite[p.~82, Corollary 2]{CLS}.
Clearly, $\tilde{P}$ is the remainder of $P$ with respect to this Gr\"obner basis, so the claim follows.
\end{proof}
Note that this theorem gives an effectively computable criterion to decide whether a Hilbert decomposition is induced by a Stanley decomposition.
\begin{rema}\label{rk:char2} Let us add some remarks.
\begin{enumerate}
\item We can say a little more about the structure of $\det A_{\ab}$.
	Endow the polynomial ring $\KK[Y_{i,j} \with (i,j) \in \tilde{I}]$ with a $\NN^{|I|}$-grading by setting $\deg Y_{i,j} := e_i$.
	It follows from the definition that the entries of a column of $A_{\ab}$ corresponding to $m_i$ are homogeneous of degree $e_i$.
	Hence $\det A_{\ab}$ is a homogeneous polynomial (with respect to this grading) and its degree is a $0/1$-vector.
	In particular, all monomials in $\det A_{a}$ are squarefree.
\item Consider the case that $\dim_\KK M_{\ab} \leq 1$ for all $\ab \in \ZZ^n$.
	Then, by the above remark, the single entry of $A_{\ab}$ is either zero or of the form $c Y_{i1}$ for some $i \in I$ and $c \in \KK\setminus\set{0}$.
	Hence the Hilbert decomposition is induced by a Stanley decomposition if and only if none of the $A_{\ab}$ is the zero matrix.
	So, in this case our Theorem \ref{thm:main} specializes to \cite[Proposition 2.8]{BKU}.
	In particular, the assumption that $\KK$ is infinite can be removed from \cite[Conjecture 5.1]{St} in the case that $M$ is an $R$-module with  $\dim_\KK M_{\ab} \leq 1$ for all $\ab \in \ZZ^n$. While this seems to be known, we could not find a precise reference for it.
\end{enumerate}
\end{rema}

	In general, the case distinction on the cardinality of the field cannot be removed.
	In fact, if $\KK$ is finite, then the condition that $\det A_{\ab} \neq 0$ for all $a$ is not sufficient.
	On the positive side, we know that the determinants of the $A_{\ab}$ are polynomials with squarefree monomials.
	Hence, if they are nonzero, then they do not vanish identically even over a finite field.
	On the other hand, it might not be possible to find values for the $Y_{i,j}$ such that all determinants are nonzero simultaneously.
	The following example shows this phenomenon.
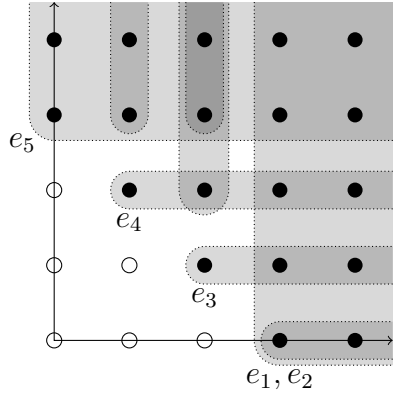
\begin{figure}[h]
\begin{tikzpicture}[scale=1, radius=0.1]
	\draw[<->] (4.5,0) -- (0,0) -- (0,4.5);

	\foreach \y in {0,1,2}
		\draw (0,\y) circle;
	\foreach \y in {3,4}
		\fill (0,\y) circle;

	\foreach \y in {0,1}
		\draw (1,\y) circle;
	\foreach \y in {2,3,4}
		\fill (1,\y) circle;

	\draw (2,0) circle;
	\foreach \y in {1,...,4}
		\fill (2,\y) circle;

	\foreach \x in {3,4}
	\foreach \y in {0,...,4}
		\fill (\x,\y) circle;

\begin{scope}[densely dotted, fill opacity=0.15]

	\begin{scope}
		\clip (-0.5,-0.5) -- (-0.5,4.5) -- (4.5,4.5) -- (4.5,-0.5) -- cycle;
		
		\filldraw (2.66,4.6) -- +(0,-4.6)
			arc[start angle=180, end angle=270, radius=0.33] -- +(1.6,0) -- (4.6,4.6) -- cycle;
			
		\filldraw (4.6, 2.66) -- +(-4.6,0)
			arc[start angle=270, end angle=180, radius=0.33] -- +(0,1.6) -- (4.6,4.6) -- cycle;
	\end{scope}

	\filldraw (1.66,4.5) -- +(0,-2.5)
		arc[start angle=180, end angle=360, radius=0.33] -- +(0,2.5);

	\filldraw (1.75,4.5) -- +(0,-1.5)
		arc[start angle=180, end angle=360, radius=0.25] -- +(0,1.5);

	\filldraw (0.75,4.5) -- +(0,-1.5)
		arc[start angle=180, end angle=360, radius=0.25] -- +(0,1.5);

	\filldraw (4.5, -0.25) -- +(-1.5,0)
		arc[start angle=270, end angle=90, radius=0.25] -- +(1.5,0);

	\filldraw (4.5, 0.75) -- +(-2.5,0)
		arc[start angle=270, end angle=90, radius=0.25] -- +(2.5,0);

	\filldraw (4.5, 1.75) -- +(-3.5,0)
		arc[start angle=270, end angle=90, radius=0.25] -- +(3.5,0);
\end{scope}

\node[below=0.24cm ] at (3,0) {$e_1, e_2$};
\node[below=0.16cm ] at (2,1) {$e_3$};
\node[below=0.16cm ] at (1,2) {$e_4$};
\node[below left=0.1cm ] at (0,3) {$e_5$};
\end{tikzpicture}
\caption{The Hilbert decomposition of Example \ref{ex:finite}.}
\label{fig:exfinite}
\end{figure}
\begin{ex}\label{ex:finite}
	Let $R = \KK[X_1, X_2]$ endowed with the standard $\mathbb{Z}^2$-grading.
	Consider the module $M$ with generators $e_1,\dots, e_5$ in degrees $(3,0),(3,0),(2,1),(1,2),(0,3)$ and relations
	\[ X_2e_1=X_1e_3,\quad  X_2^2e_2=X_1^2e_4,\quad X_2^3e_1+X_2^3e_2=X_1^3e_5.\]
	A Hilbert decomposition of $M$ is given by
	\begin{gather*}
	R_1 = \KK[X_1, X_2], \qquad R_2 = R_3 = R_4 = \KK[X_2],\\
	R_5 = \KK[X_1,X_2],\qquad R_6 = R_7 = R_8 = \KK[X_1],
	\end{gather*}
	and
	\[ \Sb_1 = \Sb_2 = (3,0), \Sb_3 = (2,1), \Sb_4 = (1,2), \Sb_5 = (0,3), \Sb_6 = (2,2), \Sb_7 = (2,3), \Sb_8 = (1,3), \]
	see Figure \ref{fig:exfinite}.
	In each degree there is a matrix $A_{\ab}$. Let us compute the matrices in the degrees $(3,1), (3,2)$ and $(3,3)$.
	For this we set $\tilde{m_1} = Y_{11} e_1 + Y_{12} e_2$, $\tilde{m}_3 = Y_3 e_3$, $\tilde{m}_4 = Y_4 e_4$ and $\tilde{m}_5 = Y_5 e_5$.
	Moreover, we choose $X_2^k e_1, X_2^k e_2$ as basis for $M_{(3,k)}$. With these conventions, we have the following matrices:
	\begin{equation}\label{eq:matrix}
		A_{(3,1)} = \begin{pmatrix} Y_{11} & Y_{3} \\ Y_{12} & 0 \end{pmatrix}
		\qquad
		A_{(3,2)} = \begin{pmatrix} Y_{11} & 0 \\ Y_{12} & Y_{4} \end{pmatrix}
		\qquad
		A_{(3,3)} = \begin{pmatrix} Y_{11} & Y_{5} \\ Y_{12} & Y_{5} \end{pmatrix}
	\end{equation}
	Their determinants are $Y_{12} Y_3$, $Y_{11} Y_4$, and $(Y_{11} + Y_{12} ) Y_5$.
	Hence over the finite field $\mathbb{F}_2$ with two elements, it is not possible to choose values $y_{11}, y_{12}$ for $Y_{11}, Y_{12}$, such that all three determinants are nonzero.
	Thus the Hilbert decomposition given above is induced by a Stanley decomposition over $\mathbb{F}_4$, say, but not over $\mathbb{F}_2$.
\end{ex}

For later use, we note the following consequence of Theorem \ref{thm:main}:
\begin{coro}\label{coro:localglobal}
	Assume that $\KK$ is infinite and let $(R_i, \Sb_i)_{i \in I}$ be a Hilbert decomposition of $M$.
	Then $(R_i, \Sb_i)_{i \in I}$ is induced by a Stanley decomposition if and only if for each $\ab \in \NN^n, \ab \preceq \gb$, there exists a linearly independent subset $(m_i)_{i \in \cC(\ab)}$ of $M_\ab$, such that $m_i \in \Xb^{\ab-\Sb_i} M_{\Sb_i}$ for $i \in \cC(\ab)$.
\end{coro}
\begin{proof}
	The condition is clearly equivalent to the non-vanishing of the determinants of $A_\ab$ for $\ab \in [\nb,\gb]$.
\end{proof}


\section{An algorithm for computing the Stanley depth of a module}\label{Algo}
In this section we describe how Theorem \ref{thm:main} can be used to effectively compute the Stanley depth of a given (finitely generated $\ZZ^n$-graded) module.
We assume (as in Section \ref{AlgStanley}) that $M$ is a fixed finitely generated $\ZZ^n$-graded $R$-module and we fix $\gb \in \NN^n$ such that $M$ is positively $\gb$-determined.

By Proposition \ref{prop:gdet}, one only needs to consider $\gb$-determined Stanley decompositions.
Hence the Stanley depth of $M$  can be expressed as
\newcommand{\Df}[1]{\mathfrak{D}(#1)}
\newcommand{\df}{\mathfrak{D}}
\[
\sdep M = \max \left\{\depth \df \with
\begin{aligned}
	&\df \text{ is a $\gb$-determined Hilbert decomposition of } M\\
	& \text{which is induced by a Stanley decomposition.}
\end{aligned}
\right\}.
\]

A key remark is that there are only \emph{finitely many} $\gb$-determined Hilbert decompositions of $M$ for a fixed $\gb$.
To actually compute the Stanley depth using this formula, one needs to
\begin{enumerate}
	\item iterate over all $\gb$-determined Hilbert decompositions $\df$ of $M$; and
	\item decide whether $\df$ is induced by a Stanley decomposition of $M$.
\end{enumerate}
\medskip

An algorithm for the first task was presented in \cite[Algorithm 1]{IZ}.
In this section we shall follow this approach and we modify \cite[Algorithm 1]{IZ}, so that it may be used for computing the Stanley depth.
We would like to remark at this point that an alternative approach for this first task is to use a description of
the set of $\gb$-determined Hilbert decompositions as the set of lattice points in a certain polytope. We give a precise description of this polytope later, in Proposition \ref{prop:dioph}.
So, in fact one may use standard software to enumerate these points, for example \texttt{SCIP} \cite{SCIP} or \texttt{Normaliz} \cite{N1, N2}.
This idea for enumerating Hilbert decompositions was originally suggested by W. Bruns and described in Katth\"an \cite[Section 7.2.1]{K}.

For the second task, we suggest to apply Theorem \ref{thm:main}. In order to make this effective, one has to choose bases for the components $M_{\ab}$ of $M$.
One possibility is to choose standard monomials with respect to some Gr\"obner bases, cf. Eisenbud~\cite[Theorem 15.3]{Eis}.
The computation of the matrices $A_{\ab}$ and their determinant can then be done using standard algorithms from constructive module theory.
We refer to Chapter 15 of \cite{Eis} or Chapter 10.4 of Becker and Weispfenning~\cite{BW}.
A possible alternative for the second task is provided by Theorem \ref{thm:stanleypolytop}.

\begin{rema}\label{rem:finiteinfinite}
	For the case distinction of Theorem \ref{thm:main}, one has to decide whether the field is finite or not.
	We describe one way to avoid this.
	With the notation introduced in Construction \ref{const}, let
		\[P=P(\mathfrak{D}) := \prod_{\ab \in [\nb,\gb]} \det A_\ab \in \KK[Y_{i,j} \with (i,j) \in \tilde{I}].\]
	If the field is finite, one has to reduce $P$ to $\tilde{P}$ as described in Theorem \ref{thm:main}, while in the infinite case one can directly use $P$.
	But even in the finite case, $P$ equals $\tilde{P}$ if the largest exponent in $P$ does not exceed the cardinality of $\KK$.
	Note that this is trivially true if $\KK$ is infinite, so we can base the case distinction on the question whether the largest exponent in $P$ exceeds the cardinality of $\KK$.
\end{rema}

\subsection{Enumerating \texorpdfstring{$\gb$}{g}-determined Hilbert decompositions via Hilbert partitions}
In the following, we present a modified version of \cite[Algorithm 1]{IZ} for the computation of the Stanley depth, see Algorithm \ref{algo:hdep} below.
Hence we obtain an algorithm for the computation of the Stanley depth of $M$.

As the algorithm in \cite{IZ} is formulated in terms of Hilbert \emph{partitions}, we recall the necessary definitions from \cite{IJ}.

Let the polynomial
\[
H_M(t)_{\preceq \gb}:=\sum_{0\preceq  \ab \preceq \gb} (\dim_\KK M_{\ab}) t^\ab
\]
be the truncated $\ZZ^n$-graded Hilbert series of $M$.
For $\ab,\bb \in \ZZ^n$ such that $\ab\preceq \bb$, we set
\[
Q[\ab,\bb](t):=\sum_{\ab \preceq \cb \preceq \bb } t^{\cb}
\]
and call it the \emph{polynomial induced by the interval} $[\ab,\bb]$.

\begin{defi}[\cite{IJ}]\label{defi:Hpartition}
We define a \emph{Hilbert partition} of the polynomial $H_M(t)_{\preceq \gb}$ to be a finite sum
\[
\mathfrak{P}: H_M(t)_{\preceq \gb}=\sum_{i \in I} Q[\ab^i,\bb^i](t)
\]
of polynomials induced by the intervals $[\ab^i,\bb^i]$.
\end{defi}
Note that there are only finitely many Hilbert partitions of $H_M(t)_{\preceq \gb}$. On one hand,
every Hilbert partition $\mathfrak{P}$ induces a $\gb$-determined Hilbert decomposition $\mathfrak{D}(\mathfrak{P})$ by the following construction.
\begin{construction}[\cite{IJ}]
\newcommand{\Gc}{\mathcal{G}}
Let $\mathfrak{P}: \sum_{i \in I} Q[\ab^i,\bb^i](t)$ be a Hilbert partition of $H_M(t)_{\preceq \gb}$.
For $\nb \preceq \ab \preceq \bb \preceq \gb$ we set
\[ \Gc[\ab,\bb] := \set{\cb \in [\ab,\bb] \with c_j = a_j \text{ for all } j \text { with } b_j = g_j}.\]
Further, for $\bb \preceq \gb$ let $Z_\bb := \set{j \in [n] \with b_j = g_j}$, $\rho(\bb)=|Z_\bb|$ and let $\KK[Z_{\bb}] := \KK[X_j \with j \in Z_\bb]$.
Then we define
\[
\mathfrak{D}(\mathfrak{P}): M\iso \bigoplus_{i=1}^r\Big(\bigoplus_{\cb\in \mathcal{G}[\ab^i,\bb^i]} K[Z_{\bb^i}](-\cb)\Big).
\]
\end{construction}
On the other hand, by Proposition \ref{lem:gstandart}, each $\gb$-determined Hilbert decomposition $(\KK[Z_i], \Sb_i)_{i \in I}$ is induced by the Hilbert partition
\[
\mathfrak{P}: H_M(t)_{\preceq \gb} = \sum_{i \in I} Q[\Sb_i,\bb^i](t),
\]
where
\[
(\bb^i)_j = \begin{cases}
(\Sb_i)_j &\text{ if } j \notin \ZZ_i,\\
g_j &\text{ if } j \in \ZZ_i.\\
\end{cases}
\]

Hence the $\gb$-determined Hilbert decompositions are exactly those Hilbert decompositions which are induced by a Hilbert partition.

Our modified version \cite[Algorithm 1]{IZ} for the computation of the Stanley depth is presented in Algorithm \ref{algo:hdep}.

\allowdisplaybreaks
\begin{algorithm}
\SetKwFunction{FindElementsToCover}{{\bf FindElementsToCover}}
\SetKwFunction{FindPossibleCovers}{{\bf FindPossibleCovers}}
\SetKwFunction{Beg}{begin}
\SetKwFunction{En}{end}
\SetKwFunction{size}{size}
\SetKwFunction{AddInterval}{{\bf AddInterval}}
\SetKwFunction{ComputeDeterminantsProduct}{{\bf ComputeDeterminantsProduct}}
\SetKwFunction{CheckStanleyDepth}{{\bf CheckStanleyDepth}}
\SetKwFunction{Reduce}{{\bf Reduce}}
\SetKwData{Boolean}{Boolean}
\SetKwData{Container}{Container}
\SetKwData{Polynomial}{Polynomial}
\SetKwData{Integer}{Integer}
\caption{Function that checks if $\sdep\ge s$ recursively} \label{algo:hdep}

\KwData{$\gb\in \NN^n$, $s \in \NN$, an $R$-module $M$, a polynomial $H(t)=H_M(t)_{\preceq \gb}\in \NN[t_1,...,t_n]$, a \Container $\mathfrak{P}$ and $q \in \NN \cup\set{\infty}$}
\KwResult{{\it true} if $\sdep M \geq s$}
\Boolean \CheckStanleyDepth{$\gb,s,M,P,\mathfrak{P},q$}\;
\Begin{
    \If {$H\notin \NN[t_1,...,t_n]$}{\Return{false}\;}
    \Container $E=$\FindElementsToCover{$\gb,s,H$}\;
\nl \If {$\size{E}=0$}{
\nl \Polynomial $P(Y)$:=\ComputeDeterminantsProduct{$\gb,M,\mathfrak{P}$}\;
\nl  P=\Reduce{$P,q$}\;
\nl \If {$P \neq 0$}{\Return{true}\;}
	\Return{false}\;}
    \Else{
    \For { i=\Beg{E} \KwTo i=\En{E} }{
    \Container $C[i]$:=\FindPossibleCovers{$\gb,s,H,E[i]$}\;
    \If {\size{$C[i]$}$=0$}{\Return{false}\;}
    \For { j=\Beg{$C[i]$} \KwTo j=\En{$C[i]$} }{
    \Polynomial $\tilde{H}(t)=H(t)-Q[E[i],C[i][j]](t)$\;
\nl \Container $\tilde{\mathfrak{P}}$:=\AddInterval{$\mathfrak{P},Q[E[i],C[i][j]](t)$}\;
    \If{\CheckStanleyDepth{$\gb,s,M,\tilde{H},\tilde{\mathfrak{P}}$}=true}{\Return{true}\;}
}
}
    \Return{false}\;
}
}
\end{algorithm}

\medskip

The differences from \cite[Algorithm 1]{IZ} appear at lines 1--4, 5, and in the usage of the extra parameters $M$, $\mathfrak{P}$, and $q$. The container $\mathfrak{P}$ is used for storing the intervals in the Hilbert partitions that have been computed, and it can be initialized empty.
The $R$-module structure of $M$ is needed for computing the matrices $A_{\ab}$ (for all $\ab \in [\nb,\gb]$). Moreover, $q$ is the cardinality of the field, which is needed for the reduction.
Assuming that the reader is familiar with \cite[Algorithm 1]{IZ},  we describe below the new key steps of the algorithm:

\begin{itemize}
	\item line~1. ~If $E$ is empty, then we have computed a complete Hilbert partition in $\mathfrak{P}$ (since there are no elements in $E$ to cover). Then we have to check using Theorem \ref{thm:main} whether the Hilbert decomposition $\mathfrak{D}(\mathfrak{P})$ is induced by a Stanley partition.
    \item line~2. The function {\bf ComputeDeterminantsProduct} computes $P(\mathfrak{D}(\mathfrak{P}))$ as in Remark \ref{rem:finiteinfinite}. Since $P$ depends on the $R$-module structure of $M$, we have to pass it as a parameter.
    \item line~3. Here we compute the reduction of $P$ with respect to the cardinality $q$ of the field.
    		We point out that we can skip this step if $\KK$ is infinite.
    \item line~4. We apply Theorem \ref{thm:main}, so we check whether $P \neq 0$. If the answer is positive, then we are done. We have reached a good leaf of the searching tree.
	\item lines~5. The child $\tilde{P}$ is generated here and further investigated in the recursive call.
\end{itemize}


\section{Applications and Examples}\label{ApplicationsExamples}
In this section, we present several applications of Theorem \ref{thm:main}.
To simplify the discussion we assume throughout this section that $|\KK|=\infty$.

\subsection{Stanley depth of syzygies}\label{subsection:counterex}
In this subsection, we present an example of an $R$-module $M$, such that $\sdep M > \sdep \Syz_R^1(M)$.
This answers Question 63 in \cite{H} to the negative.
Let us describe the idea of the construction.
It was observed in \cite{IZ} that there are modules $M$ such that $\sdep M < \sdep M \oplus R$.
But it always holds that $\Syz_R^1(M \oplus R) = \Syz_R^1(M)$.
Hence, we will look for a module whose Stanley depth increases sufficiently under adding copies of the ring, to obtain
\[
\sdep \Syz_R^1(M)=\sdep \Syz_R^1(M\oplus R^{a}) < \sdep M\oplus R^{a}.
\]

In fact, it is already sufficient to choose $M = \mm$, the maximal ideal in some polynomial ring.
It follows from \cite[Proposition 3.6]{BKU} that
\[\sdep \Syz_R^1(\mm) \leq n - \lceil\frac{n-2}{3}\rceil,¸\]
where $n$ is the number of variables. As $\sdep \mm = \lceil\frac{n}{2}\rceil$, we see that in order to use this upper bound,
we need that the Stanley depth of $\mm$ increases at least by two after adding any number of copies of the ring.
The smallest $n$ where this is possible is six. Indeed, an easy computation following Popescu \cite{PopescuA} shows that the \emph{$\ZZ$-graded Hilbert depth} of $\mm_6 \oplus R^9$ equals $5$,
while $\sdep \Syz_R^1(\mm_6\oplus R^9) = \sdep \Syz_R^1(\mm_6) \leq 6 - \lceil\frac{6-2}{3}\rceil = 4$ (see Uliczka \cite{U} for details about the $\ZZ$-graded Hilbert depth). So $M = \mm_6 \oplus R^9$ is our candidate for a counterexample.

We need to compute a Hilbert decomposition $\mathfrak{D}$ of $M$ with $\depth \mathfrak{D}=5$.
Unfortunately, this module is already too large for the CoCoA implementation of the Algorithm in \cite{IZ}.
By Proposition \ref{lem:gstandart}, it is enough to search for a $\gb$-Hilbert decomposition, where $\gb=(1,1,1,1,1,1)$.
These decompositions are described by a system of linear Diophantine inequalities (see Section \ref{ssec:rado} for details) and we can solve the system with the software \texttt{SCIP} \cite{SCIP}.
This yields the Hilbert decomposition of $M$, which is summarized in Table \ref{tab:largehdec}. There, an entry such as $2\times[001111 ,101111]$ is to be interpreted as two copies of the vector space $$\KK[X_1,X_3,X_4,X_5,X_6](0,0,-1,-1,-1,-1)$$ in the Hilbert decomposition.

\begin{table}[th]
\begin{tabular}{rrr}
  $4\times[000000 ,111110]$ & $2\times[000000 ,111101]$ & $3\times[000000 ,111011]$ \\{}
  		 $[000001 ,111101]$ & 		 $[000001 ,111011]$ & 		 $[000001 ,110111]$ \\{}
  		 $[000001 ,101111]$ & 		 $[000001 ,011111]$ & 		 $[000010 ,110111]$ \\{}
  		 $[000010 ,101111]$ & 		 $[000010 ,011111]$ & 		 $[000100 ,111101]$ \\{}
  		 $[000100 ,110111]$ & 		 $[000100 ,101111]$ & 		 $[000100 ,011111]$ \\{}
  		 $[001000 ,101111]$ & 		 $[010000 ,011111]$ & 		 $[100000 ,110111]$ \\{}
  		 $[000111 ,110111]$ & 		 $[001011 ,111011]$ & 		 $[001101 ,101111]$ \\{}
  		 $[001110 ,111110]$ & 		 $[010011 ,011111]$ & 		 $[010101 ,011111]$ \\{}
  		 $[010110 ,111110]$ & 		 $[011001 ,111101]$ & 		 $[011010 ,011111]$ \\{}
  		 $[011100 ,111101]$ & 		 $[100011 ,111011]$ & 		 $[100101 ,110111]$ \\{}
  		 $[100110 ,101111]$ & 		 $[101001 ,111101]$ & 		 $[101010 ,111110]$ \\{}
  		 $[101100 ,111101]$ & 		 $[110001 ,110111]$ & 		 $[110010 ,111110]$ \\{}
  		 $[110100 ,110111]$ & 		 $[111000 ,111011]$ & $2\times[001111 ,101111]$ \\{}
  		 $[101011 ,111011]$ & 		 $[110011 ,111011]$ & 		 $[111100 ,111110]$ \\{}
  $3\times[011111 ,011111]$ & $2\times[101111 ,101111]$ & $2\times[110111 ,110111]$ \\{}
  		 $[111011 ,111011]$ & $2\times[111101 ,111101]$ & 		 $[111110 ,111110]$\\{}
 $10\times[111111 ,111111]$ &&
\end{tabular}
\medskip
\caption{A Hilbert decomposition $\mathfrak{D}$ of $M$ with $\depth \mathfrak{D}=5$.}\label{tab:largehdec}
\end{table}

In particular, the Hilbert depth of $M$ equals $5$.
It remains to show that this Hilbert decomposition is induced by a Stanley decomposition of $M$.
Then we can conclude that
\[\sdep M = 5 > 4 \geq \sdep \Syz_R^1(M).\]

For this we prove the following general result:
\begin{propo}\label{prop:max}
	Let $\mm \subset R$ be the maximal monomial ideal. Assume that $\KK$ is infinite. Then for all $\alpha,\beta \in \NN$ it holds that
	\[ \hdep \mm^{\oplus \alpha} \oplus R^{\oplus \beta} = \sdep \mm^{\oplus \alpha} \oplus R^{\oplus \beta}. \]
	In fact, every Hilbert decomposition of this module is induced by a Stanley decomposition.
\end{propo}
\begin{proof}
	Let $M := \mm^{\oplus \alpha} \oplus R^{\oplus \beta}$.
	Further, let $e_1, \dotsc, e_\alpha, f_1, \dotsc, f_\beta$ be the natural set of generators of $R^{\oplus \alpha} \oplus R^{\oplus \beta}$ and consider $M$ as a submodule of this module.
	
	In every nonzero multidegree $\ab \in \NN^n$, the elements $\Xb^\ab e_1, \dotsc, \Xb^\ab e_\alpha, \Xb^\ab f_1, \dotsc, \Xb^\ab f_\beta$ form a vector space basis of $M_{\ab}$.
	Moreover, a vector space basis of $M_\nb$ is given by $f_1, \dotsc, f_\beta$.

	Now consider a Hilbert decomposition $(R_i, \Sb_i)_{i \in I}$ of $M$.
	We distinguish two kinds of summands in this decomposition.
	First, there are those $i$ where $\Sb_i = 0$.
	Here we set $m_i := \sum_j Z_{ij} f_j$ and we call these generators of the first type.
	As we start from a Hilbert decomposition, it is clear that there are exactly $\dim M_\nb = \beta$ generators of the first type.
	Further, for $i$ with ${\Sb_i} \neq 0$ we set $m_i := \sum_j Y_{ij} \Xb^{\Sb_i} e_j + \sum_j Z_{ij} \Xb^{\Sb_i} f_j$.
	We call these the generators of the second type.

	Next we consider the corresponding matrices as in Theorem \ref{thm:main}.
	In the multidegree $\nb$, it is easy to see that $A_\nb$ is a generic (square) matrix in the variables $Z_{ij}$, and thus its determinant is non-zero.
	So consider a multidegree $\ab \neq \nb$.
	Both types of generators can contribute to $M_{\ab}$, so the matrix $A_{\ab}$ has the following shape:

	\[ \begin{tikzpicture}
	
	\matrix[style={
		    matrix of math nodes, nodes in empty cells,
		    every node/.append style={text width=0.55cm,align=center,minimum height=5ex},
		    left delimiter=(, right delimiter=),
	}] (mat) {
	  & & & \\
	  & & & \\
	  & & & \\
	  & & & \\
	};
	\draw (mat-2-1.south west) -- (mat-2-4.south east);
	\draw (mat-1-2.north east) -- (mat-4-2.south east);
	\node[font=\large] at (mat-1-1.south east) {$0$};
	\node[font=\large] at (mat-1-3.south east) {$Y_{**}$};
	\node[font=\large] at (mat-3-1.south east) {$Z_{**}$};
	\node[font=\large] at (mat-3-3.south east) {$Z_{**}$};
	
	\draw[decoration={brace,raise=13pt},decorate]
	  (mat-1-4.north east) -- node[right=15pt] {$\alpha$} (mat-2-4.south east);
	\draw[decoration={brace,raise=13pt},decorate]
	  (mat-3-4.north east) -- node[right=15pt] {$\beta$} (mat-4-4.south east);
	\draw[decoration={brace,raise=5pt,mirror},decorate]
	  (mat-4-1.south west) -- node[below=8pt] {$u$} (mat-4-2.south east);
	
	\end{tikzpicture} \]
	Here $u$ stands for the number of generators of the first type contributing to the multidegree $\ab$.
	Note that every entry on the antidiagonal of $A_{\ab}$ is non-zero.
	Indeed, because the sum of the indices of the matrix entries is $\alpha + \beta + 1$, while for every entry of the zero-block this sum is at most $\alpha + u \leq \alpha + \beta$.
	Hence the antidiagonal gives a non-zero monomial in the Leibniz expansion of the determinant, and as all non-zero entries of the matrix are different variables, therefore cancelation cannot occur. Thus the determinant is non-zero and the claim follows from Theorem \ref{thm:main}.
\end{proof}

\begin{rema}
\begin{enumerate}
\item
Proposition \ref{prop:max} does not hold as stated for arbitrary ideals.
Consider the case $R = \KK[X_1,X_2]$ and $M = (X_1 X_2) \oplus R$.
Then $\KK \oplus X_1\KK[X_1,X_2] \oplus X_2\KK[X_1,X_2]$ is a Hilbert decomposition of $M$ that is not induced by a Stanley decomposition.

\item The result also does not hold if one adds shifted copies of the ring. Consider $R = \KK[X_1,X_2]$ and $M = (X_1,X_2) \oplus R(-1,-1)$. Then $X_1 \KK[X_1,X_2] \oplus X_2 \KK[X_1,X_2]$ is a Hilbert decomposition of $M$ which is not induced by a Stanley decomposition.
In fact, by adding shifted copies of the ring, one can always obtain a Hilbert decomposition of Hilbert depth $n$ for an arbitrary graded module $M$. For this, consider a finite free resolution of $M$,
\[ 0 \rightarrow F_p \rightarrow F_{p-1} \rightarrow \dotsb \rightarrow F_0 \rightarrow M \rightarrow 0. \]
Then the sum of the Hilbert series of the even modules equals the Hilbert series of $M$ plus the sum of  the Hilbert series of the odd modules, so the former is a Hilbert decomposition of the latter.
\end{enumerate}
\end{rema}

Based on several examples, we conjecture the following strengthening of Proposition \ref{prop:max}:
\begin{conjecture}
For every number of variables and any $\alpha,\beta \in \NN$, the $\ZZ$-graded Hilbert depth \cite{U} and the Stanley depth of $\mm^{\oplus \alpha} \oplus R^{\beta}$ coincide.
\end{conjecture}


\subsection{The set of \texorpdfstring{$\gb$}{g}-determined Stanley decompositions}\label{ssec:rado}
In this section, we show that the set of all $\gb$-determined Stanley decompositions can be described by a (large) system of linear Diophantine inequalities, or, equivalently, by the set of $\ZZ^n$-lattice points inside a polytope $\mathscr{P}$.

Consider a finitely generated $\NN^n$-graded $R$-module $M$ which is $\gb$-determined for some $\gb \in \NN^n$.
Let
\[ \Omega := \set{ (\KK[Z], \ab) \with \ab \in \NN^n, \ab \preceq \gb, Z \subseteq [n], \set{j \with g_j = a_j} \subseteq Z } \]
be the set of all possible building blocks for a $\gb$-determined Hilbert decomposition of $M$ (according to Proposition \ref{lem:gstandart}).

We write $\NN^\Omega$ for the free commutative monoid with generators $\set{\eb_\omega \with \omega \in \Omega}$.
A $\gb$-determined Hilbert decomposition $(R_i, \Sb_i)_{i\in I}$ of $M$ can then be identified with the element $\sum_{i\in I} e_{(R_i, \Sb_i)} \in \NN^\Omega$.
For a vector $u \in \NN^\Omega$, we write $u(\ab,Z)$ for the component of $u$ corresponding to $Z \subseteq [n]$ and $\ab \in [\nb,\gb]$.
Note that for $(\KK[Z], \bb) \in \Omega$ and $\ab \succeq \bb$, it holds that $(\KK[Z](-\bb))_\ab \neq 0$ if and only if $\supp(\ab-\bb) \subseteq Z$. Now, $\gb$-determined Hilbert decompositions may be characterized easily.
\begin{propo}\label{prop:dioph}
A vector $u \in \NN^\Omega$ corresponds to a $\gb$-determined Hilbert decomposition of $M$ if and only if it satisfies the following equalities:
\begin{align}\label{eq:hilb}
	\sum_{\bb \in [\nb,\ab]} \sum_{\substack{Z \subseteq [n] \\ \supp(\ab-\bb) \subseteq Z}} u(\bb,Z) &= \dim_\KK M_\ab & \text{ for } \ab \in [\nb,\gb].
\end{align}
\end{propo}

So, the set of $\gb$-determined Hilbert decompositions corresponds naturally to the set of $\ZZ^n$-lattice points in the polytope $\mathscr{H}$ of non-negative solutions to \eqref{eq:hilb}.
The set of \emph{$\gb$-determined Stanley decompositions} is a subset of this.
By the following result, this subset may be defined by linear inequalities as well, i.e. the $\gb$-determined Hilbert decomposition of $M$ which are induced by $\gb$-determined Stanley decompositions correspond to the $\ZZ^n$-lattice points in a certain polytope $\mathscr{P}$.
This is the main result of this subsection.

\begin{theo}\label{thm:stanleypolytop}
A vector $u \in \NN^\Omega$ corresponds to a $\gb$-determined Hilbert decomposition of $M$ which is induced by a $\gb$-determined Stanley decomposition, if and only if it satisfies both \eqref{eq:hilb} and in addition the following inequalities:
	\begin{align}\label{eq:stan}
	\sum_{\bb \in J} \sum_{\substack{Z \subseteq [n] \\ \supp(\ab-\bb) \subseteq Z}} u(\bb,Z) &\leq  \dim_\KK \sum_{\bb \in J} \Xb^{\ab-\bb} M_{\bb} & \text{ for } \ab \in [\nb,\gb], J \subseteq [\nb,\ab].
	\end{align}
Here, the sum on the right-hand side is a sum of vector spaces.
\end{theo}

\begin{rema}
	The system of inequalities \eqref{eq:stan} is rather large, so it does not seem to be feasible for the actual computation of the Stanley depth.
	However, the theorem shows that the set of all Stanley decomposition has a nice structure.
	Note that the integer solutions of \eqref{eq:stan} for a fixed $a \in [\nb, \gb]$ form a discrete polymatroid, cf.~Herzog and Hibi \cite{HH2}.
	So the set of $\gb$-determined Stanley decompositions may also be seen as an intersection of discrete polymatroids with the polytope $\mathscr{H}$.
\end{rema}

The proof uses Rado's theorem, which we recall for the reader's convenience.
Recall that a \emph{transversal} of a set system $A_1, \dotsc , A_r$ is a collection of pairwise different elements $a_1 \in A_1, a_2 \in A_2, \dotsc, a_r \in A_r$.
\begin{theo}[Rado's theorem, VIII.2.3 \cite{Aigner}]\label{thm:rado}
Let $M$ be a matroid on a ground set $B$ with rank function $r$ and let $\mathfrak{A}: A_1, \dotsc , A_r \subseteq B$ be a collection of subsets of $B$.
Then $\mathfrak{A}$ has an independent transversal if and only if
\[ |I| \leq r\left(\bigcup_{i \in I} A_i\right) \]
for every subset $I \subset [r]$.
\end{theo}

\newcommand{\Vc}{\mathcal{V}}
\noindent We use the following variant of Rado's theorem.
\begin{coro}\label{cor:rado}
Let $V$ be a vector space and $\Vc: V_1, \dotsc, V_s$ a collection of linear subspaces of $V$.
For $u \in \NN^s$, the following are equivalent:
\begin{enumerate}
	\item There exists an independent transversal of $\Vc$, i.e. a linearly independent family of vectors $v_1 \in V_1, v_2 \in V_2, \dotsc, v_s \in V_s$.
	\item For each subset $I \subseteq \set{1, \dots, s}$, the following inequality holds:
		\[ |I| \leq \dim_\KK \sum_{i \in I} V_i. \]
		Here, the sum on the right-hand side is a sum of vector spaces.
\end{enumerate}
\end{coro}
\begin{proof}
	The inequality is clearly necessary, so we only need to show the sufficiency.
	Let $A_i$ be a basis for $V_i, 1 \leq i \leq s$.
	Consider the union $M := \bigcup_i A_i$ as a matroid.
	By Rado's theorem \ref{thm:rado}, $A_1, \dotsc, A_s$ has an independent transversal if and only if
	\[ |I| \leq \dim_\KK \lin\left(\bigcup_{i \in I} A_i\right) = \dim_\KK \sum_{i \in I} V_i \]
	for every subset $I \subset [s]$.
	Hence the inequality in our claim is sufficient.
\end{proof}

\begin{proof}[Proof of Theorem \ref{thm:stanleypolytop}]
	Assume that $u \in \NN^\Omega$ is indeed a $\gb$-determined Hilbert decomposition of the module $M$.
	By Corollary \ref{coro:localglobal}, the Hilbert decomposition $u$ corresponds to a Stanley decomposition
	if and only if for each $\ab \in [\nb,\gb]$, there are linearly independent elements $( m(\bb,Z,i) )_{(\bb,Z,i) \in \Lambda}$, such that $m(\bb,Z,i) \in \Xb^{\ab-\bb} M_\bb$ for all $(\bb,Z,i) \in \Lambda$, where
	\[ \Lambda := \Lambda(\ab,u) := \set{(\bb,Z,i) \with \bb \in [\nb,\ab], Z \subset [n], \supp(\ab-\bb) \subset Z, 1\leq i \leq u(\bb,Z)}. \]
	So in particular, the inequality in our claim is necessary.
	
	We apply the preceding Corollary \ref{cor:rado} to the vector space $M_\ab$ and the collection $(\Xb^{\ab-\bb} M_\bb)_{(\bb,Z,i) \in \Lambda}$ of subspaces.
	For a subset $I \subset \Lambda$, consider
	\[ \bar{I} := \set{(\bb,Z,i) \in \Lambda \with (\bb,Z', i') \in I \text{ for some } Z', i'}. \]
	It clearly holds that
	\[ \sum_{(\bb,Z,i) \in I} \Xb^{\ab-\bb} M_\bb = \sum_{(\bb,Z,i) \in \bar{I}} \Xb^{\ab-\bb} M_\bb, \]
	hence it suffices to consider subsets of the form $\bar{I}$, and these are in bijection with subsets $J \subseteq [\nb,\ab]$.
	Hence our inequalities are also sufficient.
\end{proof}

\section*{Acknowledgements}
The authors are greatly indebted to J\"urgen  Herzog for bringing Question \ref{Q:Herzog63} to our attention.

\bibliographystyle{alpha}
\bibliography{LCM2}

\end{document}